\documentclass[twoside,11pt]{article}

\usepackage{amsmath}
\usepackage{amssymb}
\usepackage{latexsym,amsthm}
\usepackage{mathtools}
\usepackage[final]{showkeys}
\usepackage{graphicx}
\usepackage{mathrsfs}
\usepackage{xcolor}
\usepackage{hyperref}
\usepackage{enumitem}
\usepackage{esint}

\newtheorem{lemma}{Lemma}
\newtheorem{corollary}{Corollary}
\newtheorem{theorem}{Theorem}

\theoremstyle{definition}
\newtheorem{remark}{Remark}
\newtheorem{definition}{Definition}

\title{
Global minimizers for  a two-sided biharmonic 
Alt-Caffarelli 
problem
}

\author{Hans-Christoph Grunau
\thanks{Fakult\"{a}t f\"{u}r Mathematik, 
Otto-von-Guericke-Universit\"{a}t,
Postfach 4120,
39016 Magdeburg, Germany,
E-mail: \texttt{hans-christoph.grunau@ovgu.de}}
\and 
Marius M\"uller
\thanks{Institut f\"ur Mathematik, Universit\"at Augsburg,
	Universit\"atsstra{\ss}e 2,
	86159 Augsburg, Germany,
E-mail: \texttt{marius1.mueller@uni-a.de}
 \\
 \emph{Keywords:} 
Alt-Caffarelli problem, Biharmonic operator, Homogeneous global minimizers, PDEs with distributions which are not measures. 
 \\
 \emph{MSC 2020 Subject classification:} 35R35, 31830 (Primary), 49J52, 35B44 (Secondary)
 }
}

\date{\today}

\begin{document}

\maketitle

\begin{abstract}
    We study global minimizers of biharmonic analogues of the Alt-Caffarelli functional. It turns out that half-space solutions are global minimizers for the two-sided Alt-Caffarelli functional, but not in the one-sided case. In addition, we identify a further class of global minimizers, all of which have constant Laplacian. Recent work by J. Lamboley and M. Nahon reduces potential global minimizers in dimension two to four possible categories.
    Our work shows that three of these categories persist in any dimension and are in fact global minimizers. 
\newline   
Moreover, we show that minimizers of the two-sided biharmonic Alt-Caffarelli problem do in general not satisfy a partial differential equation, not even with a signed measure as right-hand-side. This is in sharp contrast to the corresponding one-sided problem.
\end{abstract}



\section{Introduction}

Let $\Omega \subset \mathbb{R}^n$ be open and bounded. For $u \in H^2(\Omega)$ we define the \emph{two-sided biharmonic Alt-Caffarelli functional}
\begin{equation}\label{eq:bihAC}
    \mathcal{F}(u,\Omega) := \int_\Omega (\Delta u)^2 \; \mathrm{d}x + |\{x \in \Omega : u(x) \neq 0 \}|.
\end{equation}
In the following we will write simply $|\{u \neq 0 \}|$ instead of $|\{ x \in \Omega : u(x) \neq 0 \}|$ if $\Omega$ is clear from the context. Minimizers of the functional have recently been studied in \cite{GMMathAnn} under Dirichlet and Navier boundary conditions. It turned out that for any 
boundary datum $u_0$ and any $C^2$-domain $\Omega \subset\mathbb{R}^n$, minimizers of $\mathcal{F}(\cdot,\Omega)$ lie in $C^{1,\alpha}(\Omega)$ for any $\alpha \in (0,1)$, but in general not in $C^2(\Omega)$.
The question of $C^{1,1}$-regularity remains open, but symmetry results and explicit study of radial minimizers suggest a positive answer. In dimension $n = 2$ the behaviour  of minimizers at the free boundary $\partial (\{u\not=0\}\cup\{\nabla u\not=0\})$ has been addressed in the recent work \cite{Nahon}.

The main goal of this article is to study \emph{global minimizers} of the above functional in arbitrary dimensions. These are defined as follows.
\begin{definition}\label{eq:globalmibn}
    We say that $u \in H^2_{loc}(\mathbb{R}^n)$ \emph{minimizes $\mathcal{F}$ globally
    } 
    if for all bounded open sets $\Omega \subset \mathbb{R}^n$ and all $v \in H^2(\Omega)$ such that $v- u \in H_0^2(\Omega)$ one has $\mathcal{F}(u,\Omega) \leq \mathcal{F}(v,\Omega)$. 
\end{definition}

\noindent
This definition may be equivalently reformulated as follows: 
\begin{quote}
    We say that $u \in H^2_{loc}(\mathbb{R}^n)$ \emph{minimizes $\mathcal{F}$ globally} 
if for all bounded open sets $\Omega \subset \mathbb{R}^n$ and all $v \in  H^2_{loc}(\mathbb{R}^n)$ such that $v= u $ outside $\Omega$ one has $\mathcal{F}(u,\Omega) \leq \mathcal{F}(v,\Omega)$. 		
\end{quote}

We find examples of global minimizers and in particular obtain some 
\emph{two-homogeneous} global 
minimizers. 
Throughout this paper, ``homogeneous'' always means ``positive homogeneous''.
The absence of a maximum principle imposes a challenge in this context, especially since uniqueness and symmetry results are not readily available. Examples of global minimizers that we find fall under two categories
\begin{itemize}
    \item[(I)] Half-space minimizers of the form $u(x) = \pm \frac{1}{2}((x-x_0) \cdot \nu)^2 \chi_{(0,\infty)}((x-x_0) \cdot \nu)$ for some unit vector $\nu \in \mathbb{R}^n$ and $x_0 \in \mathbb{R}^n$.
    See Theorem~\ref{thm:Main} in Section~\ref{sec:halfplane}. 
    \item[(II)] Functions $u$ with constant Laplacian, i.e $\Delta u= c$ for some $c \in \mathbb{R}$. However, global minimality of such functions depends on the value of the constant $c$. The condition $|c| \geq 1$ ensures global minimality, see Theorem~\ref{lem:globmin} in Section~\ref{sec:constLaplacian}.  
    For other constants, global minimiality is in general not  
    true, as we shall show in Lemma~\ref{lem:lemma8}. 
\end{itemize}
Both categories contain two-homogeneous functions and hence give rise to two-homogeneous global minimizers. The above list is not necessarily an exhaustive list of global minimizers. The question whether more global minimizers exist arises naturally. To this end one may see e.g. the recent work \cite{Nahon} by Lamboley and Nahon, where a potential ``angular minimizer'' shows up in dimension $n=2$.

	The present work also pursues two ideas how the above list could be extended. Firstly, we investigate whether the condition $|c| \geq 1$ in Type (II) can be relaxed or removed, revealing that it can not be removed completely. Secondly, we show that the prefactor  $\pm \frac{1}{2}$ in Type (I) cannot be replaced by any other number. To this end note that in the special case $\nu= e_n := (0,0,...,1), x_0=0$ in Type-(I) minimizers take the form $u(x)= \pm \frac{1}{2} x_n^2 \chi_{(0,\infty)}(x_n)$. Inspired by this structure we compute all $d_1,d_2 \in \mathbb{R},d_1\neq d_2$ such that the functions 
\begin{equation}\label{eq:2A}
    w_{d_1,d_2}: \mathbb{R} \to \mathbb{R}, \quad w_{d_1,d_2}(x) :=d_1 \tfrac{x_n^2}{2} \chi_{(0,\infty)}(x_n) + d_2 \tfrac{x_n^2}{2} \chi_{(-\infty,0)}(x_n)
\end{equation}
yield global minimizers and find that necessarily $d_1= \pm 1,d_2=0$ or $d_1=0,d_2= \pm1$ (cf. Section \ref{sec:wd1d2}). Notice that there we also study the omitted case $d_1=d_2$, but minimizers that arise in this case would fall under Category (II) rather than Category (I).

 Explicit knowledge of global minimizers provides
 more insights into the 
 \emph{blow-up limits} of the minimizers found in \cite{GMMathAnn} around their free boundary (despite the fact that the absence of a maximum principle makes a rigorous result in this situation 
 difficult). Such blow-up limits are used in the classical Alt-Caffarelli problem (see e.g. \cite[Theorem 1.2 and the outline of its proof]{Velichkov})
 to study regularity of the free boundary. We will elaborate on this connection more precisely below (cf. Section \ref{subsec:appl}).   
 
The lecture notes  \cite{Velichkov} provide a broad survey on the  regularity theory of free boundaries of the classical Alt-Caffarelli problem \cite{AltCaffarelli}. 
One studies for a given smooth domain $G \subset \mathbb{R}^n$ minimizers of 
\begin{equation}\label{eq:ACclassical}
    \mathcal{F}_{AC}(u) := \int_G |\nabla u|^2 \; \mathrm{d}x + |\{ x \in G:  u(x) \neq 0 \}| \qquad \textrm{among all $u \in W^{1,2}(G): u \vert_{\partial G} = 1$}
\end{equation}
and seeks to understand properties of 
the free boundary $\partial \{u \neq 0 \}$. Notice that due to the maximum principle, minimizers of $\mathcal{F}_{AC}$ satisfy $u \geq 0$ and hence one may replace $\{ u \neq 0\}$ by $\{ u> 0\}$. For the higher order functional in \eqref{eq:bihAC}, replacing $\{ u \neq 0\}$ by $\{ u> 0\}$ makes a significant difference, as we shall see below. This is why we distinguish between the \emph{two-sided problem}  \eqref{eq:bihAC} and the \emph{one-sided problem}, where  the functional 
\begin{equation*}
    \tilde{\mathcal{F}}(u,\Omega)= \int_\Omega (\Delta u)^2 \; \mathrm{d}x + | \{x \in \Omega : u(x) > 0 \}|. 
\end{equation*}
is investigated. 
Our study here is focussed on global minimizers of 
the two-sided problem, but we also prove in Lemma \ref{lem:onesided} that these are not global minimizers of the one-sided problem.

Recently, the functionals $\mathcal{F}$ and $\tilde{\mathcal{F}}$ 
have raised a lot of interest. In \cite{DKV19,DKV20}, the one-sided problem is
studied in arbitrary dimensions with Navier boundary conditions. The authors show existence and first regularity results for minimizers and give a dichotomy result for free boundary points. Not only the pure one-sided functional is studied, but also the one-sided functional with the additional obstacle constraint $u \geq 0$. Notice that under such an 
obstacle constraint, the one-sided and two-sided functionals coincide. Such \emph{adhesive obstacle problems} have been studied before in one dimension for curves \cite{Miura1, Miura2} and are of interest in many biological applications, such as \emph{gecko adhesion}, cf. \cite{PerssonGorb}. 

 For the one-sided functional without the obstacle constraint it has been shown in \cite{MuellerAMPA,MuellerPolyharmonic} that in two dimensions minimizers of the one-sided problem enjoy optimal $C^{2,1}$-regularity in two dimensions. The same regularity result applies to the free boundary $\partial \{ u \neq 0 \}$ since it is observed that the gradient of minimizers can not have zeroes on the free boundary (which intertwines the regularity of minimizers and free boundary). These observations do not change if one replaces the Laplacian by an anisotropic elliptic operator of the form $\mathrm{div}(A\nabla)$, as \cite{MuellerAnisotropic} shows. In \cite{GMMathAnn}, the two-sided problem is studied for the first time. The different gradient behaviour yields a different optimal regularity result (under Navier boundary conditions): $C^2$-regularity is impossible. The origin of this discrepancy is also discussed in this article: In contrast to the one-sided problem, minimizers of the two-sided functional may in general not satisfy any measure-valued PDE, not even with a signed measure. This is also a remarkable difference to the two-sided biharmonic obstacle problem \cite{TwoSidedObstacle}, where a measure-valued PDE is indeed satisfied with a signed measure. Notice that for one-sided problems, measure-valued PDEs arise naturally from variational inequalities, as e.g. observed in the classical biharmonic obstacle problem \cite{CaffarelliFriedman}.

\subsection{Application: Blow-Up-Analysis}\label{subsec:appl}

To understand the regularity of the free boundaries of minimizers $\tilde{u}\in W^{1,2}(G)$ of \eqref{eq:ACclassical}, a classical approach discussed in \cite{Velichkov} works via a \emph{blow-up analysis}. More precisely, for some point $x_0$ on the free boundary $\partial \{ \tilde{u} \neq 0 \}$ one forms 
\begin{equation}
    \tilde{u}_{x_0,r}(x) := \frac{\tilde{u}(x_0 + r x)}{r}
\end{equation}
and studies 
the limit as $r \searrow 0$. 
Limits of these functions have (if existent in a suitable sense) naturally the following properties:
\begin{itemize}
    \item They are global minimizers of the classical Alt-Caffarelli functional and 
    \item they are one-homogeneous functions. 
\end{itemize}
Regularity results of the free boundary are then obtained by studying the zero level set of such one-homogeneous global minimizers. In dimension two and assuming $x_0=0$ 
 nonnegative one-homogeneous global minimizers are exhaustively classified by half-plane solutions of the form 
\begin{equation*}
    \tilde{u}_\nu(x) := (x \cdot \nu) \chi_{(0,\infty)}(x\cdot \nu),
\end{equation*}
where $\nu \in \mathbb{R}^2$ is an arbitrary unit vector (cf. \cite[Proposition 9.13]{Velichkov}). An inspection of this proof reveals that if one drops the prerequisite of nonnegativity, also affine linear functions can be one-homogeneous global minimizers of $\mathcal{F}_{AC}$.     
The free boundary  $\partial\{\tilde{u}_\nu>0\}$ is always a straight line 
of the form $\{x \in \mathbb{R}^2: x \cdot \nu = 0 \}$ i.e. a smooth one-dimensional manifold. This smoothness of the free boundary of blow-ups is decisive for the regularity of the free boundary of the original problem. The same approach works also in higher dimensions, and it turns out that the regularity of the free boundary of one-homogeneous global minimizers is decisive for regularity results. This is why \cite{Velichkov} defines the critical dimension $d^*$ of the Alt-Caffarelli problem as follows
\begin{definition}[{\cite[Definition 1.5]{Velichkov}}]\label{eq:defcriticaldim} $d^*$ is the smallest dimension such that there exists a function $z: \mathbb{R}^{d^*} \rightarrow \mathbb{R}$ such that 
\begin{itemize}
    \item[(i)] $z$ is nontrivial, non-negative and one-homogeneous,
    \item[(ii)] $z$ is a global minimizer of the classical Alt-Caffarelli functional,
    \item[(iii)] the 
    free boundary $\partial \{ z > 0 \}$ is not a $(d^*-1)$-dimensional $C^1$-submanifold of $\mathbb{R}^{d^*}$. 
\end{itemize}
\end{definition}
It is now known that $d^* \in \{ 5,6 \}$ (cf. \cite{Velichkov,dstar1,dstar2,dstar3}). 
It seems meaningful to think about a generalization of this approach to the biharmonic Alt-Caffarelli problem for $\mathcal{F}(\cdot, \Omega)$. Notice first that if $x_0$ is a free boundary point, then in view of Lemma  \ref{leminteriorregularity} one has $u(x_0)=0, \nabla u(x_0)= 0$. Hence it seems more meaningful to study the following blow-up of a minimizer $u$
\begin{equation}
    u_{x_0,r}(x) := \frac{u(x_0+ rx)}{r^2}. 
\end{equation}
Due to the lack of the maximum principle for a higher order problem it is generally not easy to show that $u_{x_0,r}$ possesses a limit as $r \searrow 0$. In dimension $n= 2$ an existence result for limits has been achieved by J. Lamboley and M. Nahon in \cite{Nahon}, by means of a monotonicity formula based on the one in \cite{DKV20}. 
The authors use this to describe the free boundary and to classify its potential singular points, if $n=2$.
In higher dimensions the existence question of such limits seems challenging and goes beyond the scope of this article. However, if existent, such limits are naturally two-homogeneous.


Moreover, cut-and-paste-arguments suggest again to expect that blow-up limits are global minimizers of $\mathcal{F}$. Observe that the condition $v-u \in H_0^2(\Omega)$ in Definition \ref{eq:globalmibn} appears here naturally, since it is needed for cut-and-paste procedures. As a result, blow-up limits are indeed expected to be two-homogeneous global minimizers.

Definition \ref{eq:defcriticaldim} only admits nonnegative functions, which is  not meaningful in the biharmonic context due to the lack of a maximum principle.  
We leave out this nonnegativity requirement but remark at this point that 
 the decisive counterexample is still nonnegative, cf. Corollary \ref{cor:dstarstar}.

Based on the reasoning above, one may define
the critical dimension $d^{**}$ for the biharmonic Alt-Caffarelli problem as follows. 

\begin{definition}\label{eq:defcriticaldimbih} $d^{**}$ is the smallest dimension such that there exists a function $z: \mathbb{R}^{d^{**}} \rightarrow \mathbb{R}$ such that 
\begin{itemize}
    \item[(i)] $z$ is nontrivial and two-homogeneous,
    \item[(ii)] $z$ is a global minimizer of $\mathcal{F}$,
    \item[(iii)] the 
    free boundary $\partial \{ z = 0, \nabla z = 0 \}$ is not a $(d^{**}-1)$-dimensional $C^1$-submanifold of $\mathbb{R}^{d^{**}}$. 
\end{itemize}
\end{definition}

We will study the value of $d^{**}$ and find in Corollary \ref{cor:dstarstar} that $d^{**} = 2.$ The reason for this is readily understood: The class of Type-(II)-minimizers contains many quadratic forms and the free boundaries of quadratic forms (in the sense of Definition \ref{eq:defcriticaldimbih}(iii)) can be more complicated than just a $C^1$-hypersurface. For example they can consist only of a singleton or of a subspace whose codimension is higher than one. 
An interesting question for future research is whether even more complicated free boundaries than singletons or lower-dimensional subspaces can arise as free boundaries of two-homogeneous global minimizers. It is conjectured that even in dimension two they do arise. Indeed, in \cite[Eq. (1.8)]{Nahon} the authors find a candidate for a two-homogeneous global minimizer whose free boundary consists of a cone, whose opening angle $t_1$ is the unique solution of $\tan(t_1)=t_1$ in $(\pi,2\pi)$. Its global minimality is conjectured and not proven so far.


As a consequence, the study of the regularity of the free boundary via blow-up in the biharmonic setting is more involved than  in the first order 
 setting. This is not just due to unsettled convergence questions in dimension $n \geq 3 $ (e.g. since homogeneity results of blow-ups are more involved than the celebrated approach in \cite{Weiss})  but also due to different qualitative behaviour of blow-up limits.  

\subsection{Structure of the article}

In Section \ref{sec:Prelim} we discuss some properties of global minimizers. 
In Section \ref{sec:halfplane} we prove global minimality of half-plane minimizers (Type I) and discuss why this global minimality is not true for the one-sided functional. In Section \ref{sec:constLaplacian} we prove the existence of global minimizers with constant Laplacian (Type II) and discuss which values of the Laplacian ensure global minimality.  In Section \ref{sec:wd1d2} we examine global minimality of a class of functions that generalizes Type I and Type II in order to initiate a discussion about how exhaustive our obtained list  of global minimizers is. 
Finally, we discuss in Section \ref{sec:NoMeasValuedInequality} that 
minimizers do in general \emph{not} satisfy a partial differential equation with  measure-valued right-hand side.

\section{Preliminaries}\label{sec:Prelim}

In the sequel we fix some notation and collect some basic facts about global minimizers.

Throughout the article, for any set $E \subset \mathbb{R}^n$, $\chi_E$ denotes the characteristic function of $E$ and $e_1,...,e_n$ denote the canonical basis vectors in $\mathbb{R}^n$. Moreover $B_R(0)$ denotes the ball in $\mathbb{R}^n$ with radius $R$ centered at zero.

\subsection{Properties of global minimizers}

In the following we will frequently use a regularity result for minimizers of $\mathcal{F}(\cdot ,\Omega)$ from \cite[Theorem 2]{GMMathAnn}.

\begin{lemma}[{\cite[Theorem 2]{GMMathAnn}}]\label{leminteriorregularity}
    Let $\Omega \subset \mathbb{R}^n$ be open and $u_0 \in C^\infty(\Omega)$. Further let $u \in H^2(\Omega)$ be a minimizer of $\mathcal{F}(\cdot, \Omega)$ in $\mathscr{D}(u_0):= \{ v \in H^2(\Omega) : v -u_0 \in H_0^2(\Omega) \}.$ Then $u \in C^{1,\alpha}(\Omega)$ and $u$ is smooth and biharmonic on $\Omega \setminus \{ u = 0, \nabla u = 0 \}$. 
\end{lemma} 
Note in particular that regularity can only break down at points where $u = 0$ and $\nabla u = 0$, i.e. non-regular points of the zero level set. As a consequence it seems more suitable to think of the free boundary as $\partial \{u = 0, \nabla u = 0 \}$ and not as $\partial \{ u = 0 \}$. 
Next we deal with symmetries of global minimizers. To this end we first observe the following symmetries of the functional $\mathcal{F}(\cdot, \Omega)$, which can be verified by direct computation:
\begin{itemize}
    \item For each $u \in H^2(\Omega)$ one has $\mathcal{F}(-u,\Omega)= \mathcal{F}(u,\Omega)$.
    \item  For each $u \in H^2(\Omega)$ and $O \in O(n,\mathbb{R})$ one has $\mathcal{F}(u \circ O,\Omega)= \mathcal{F}(u,\Omega)$.
\end{itemize}

\begin{lemma}\label{lem:invariance}
    Let $u \in H^2_{loc}(\mathbb{R}^n)$ be a global minimizer of $\mathcal{F}$. Then it holds:
    \begin{itemize}
        \item[(a)] $-u \in H^2_{loc}(\mathbb{R}^n)$ is also a global minimizer of $\mathcal{F}$.
        \item[(b)] For each $O \in O_n(\mathbb{R})$ the function $u \circ O \in H^2_{loc}(\mathbb{R}^n)$ is also global minimizer of $\mathcal{F}$.
        \item[(c)] For each $x_0 \in \mathbb{R}^n$, $u(\, \cdot\, + x_0)$ is also a global minimizer of $\mathcal{F}$.
    \end{itemize}
\end{lemma}
\begin{proof}
    We only prove (a). Assertions  (b) and (c) follow along the same lines. Suppose that $\Omega \subset \mathbb{R}^n$ 
    is open and bounded and $v \in H^2(\Omega)$ is such that $v- (-u) \in H_0^2(\Omega)$. One readily checks that then $-v \in H^2(\Omega)$ and $(-v)-u \in H_0^2(\Omega)$. Since $u$ is a global minimizer this implies $\mathcal{F}(-v, \Omega)\geq \mathcal{F}(u,\Omega)$. As a result we have
   \begin{equation*}
     \mathcal{F}(v,\Omega)=   \mathcal{F}(-v, \Omega) \geq \mathcal{F}(u,\Omega) = \mathcal{F}(-u,\Omega). 
   \end{equation*}
\end{proof}
These symmetry results allow us to reduce the functions in Type I to one \emph{prototype}. Indeed by a sign change we may assume that $\pm = +$ and by a rotation we may achieve $\nu = e_n$. As a result each function in Type I coincides up to rotation, translation,
and sign change with 
\begin{equation}\label{eq:2}
    u : \mathbb{R}^n \rightarrow \mathbb{R}, \quad u(x):= \begin{cases}
        \frac{1}{2}x_n^2 & x_n > 0,\\ 0 & x_n \leq 0.
    \end{cases}
\end{equation}
Hence  it is sufficient to show that this function yields a global minimizer of $\mathcal{F}.$ We next simplify the problem further. In Definition \ref{eq:globalmibn}, minimality must be checked for all bounded open sets $\Omega \subset \mathbb{R}^n$. We claim next, that it is sufficient to check minimality in the special case that $\Omega$ is a ball centered in the origin. In order to show this we need the following extension lemma. 

\begin{lemma}\label{lem:3}
    Suppose that $\Omega, \tilde{\Omega} \subset \mathbb{R}^n$ are bounded open sets such that $\tilde{\Omega} \supset \Omega$. Let $u \in H^2(\tilde{\Omega})$ and $v \in H^2(\Omega)$ be such that $v- u \in H_0^2(\Omega)$. Define $\tilde{v} : \tilde{\Omega} \rightarrow \mathbb{R}$ via 
    \begin{equation*}
        \tilde{v}(x) = \begin{cases}
             v(x) & x \in \Omega \\ u(x) & x \in \tilde{\Omega} \setminus \Omega
        \end{cases}
    \end{equation*}
    Then $\tilde{v} \in H^2(\tilde{\Omega})$. Moreover $\tilde{v}- u \in H_0^2(\tilde{\Omega})$. 
\end{lemma}
\begin{proof} The proof is immediate, observing that $C_0^\infty(\Omega)\subset H_0^2(\Omega)$ is dense.
\end{proof}

Now we claim that minimality in balls centered at zero is sufficient for global minimality. 

\begin{lemma}\label{lem:BR(0)istgenug}
Suppose that $u \in H^2_{loc}(\mathbb{R}^n)$ is such that for all $R> 0$ the following is satisfied: For each $v \in H^2(B_R(0))$ such that $v-u \in H_0^2(B_R(0))$ one has $\mathcal{F}(v,B_R(0)) \geq \mathcal{F}(u,B_R(0)).$ Then $u$ is a global minimizer of $\mathcal{F}$.  
\end{lemma}
\begin{proof}
    Let $\Omega \subset \mathbb{R}^n$ be open and bounded. Then there exists $R> 0$ such that $\Omega \subset B_R(0)$. Now let $v \in H^2(\Omega)$ be such that $v- u \in H_0^2(\Omega)$.
    We can define $\tilde{v}: B_R(0) \rightarrow \mathbb{R}$ to be 
    \begin{equation*}
        \tilde{v}(x) = \begin{cases}
             v(x) & x \in \Omega \\ u(x) & x \in B_R(0) \setminus \Omega
        \end{cases}
    \end{equation*}
   and conclude by Lemma \ref{lem:3} that $\tilde{v} \in H^2(B_R(0))$ and $\tilde{v} - u  \in H_0^2(B_R(0))$. By assumption we conclude 
   \begin{equation}\label{eq:22}
       \mathcal{F}(\tilde{v}, B_R(0)) - \mathcal{F}(u, B_R(0)) \geq 0.
   \end{equation}
   Since $\tilde{v}= u$ on $ B_R(0) \setminus \Omega$ we have $\mathcal{F}(\tilde{v}, B_R(0) \setminus \Omega) = \mathcal{F}(u, B_R(0) \setminus \Omega)$ and as a consequence
   \begin{equation*}
        \mathcal{F}(\tilde{v}, B_R(0)) - \mathcal{F}(u, B_R(0)) = \mathcal{F}(v, \Omega) - \mathcal{F}(u, \Omega).
   \end{equation*}
   This and \eqref{eq:22} yield $\mathcal{F}(v, \Omega) \geq \mathcal{F}(u, \Omega)$. The claim follows. 
\end{proof}

\section{Type I: Half-space minimizers}\label{sec:halfplane}

The main result of this section as follows.

\begin{theorem}\label{thm:Main}
    Let $u \in H^2_{loc}(\mathbb{R}^n)$ be defined as in \eqref{eq:2}:
    $$
     u : \mathbb{R}^n \rightarrow \mathbb{R}, \quad u(x):= \begin{cases}
    	\frac{1}{2}x_n^2 & x_n > 0,\\ 0 & x_n \leq 0.
    	\end{cases}
    	$$
     Then $u$ minimizes $\mathcal{F}$ globally. 
\end{theorem}

Recall that by Lemma \ref{lem:BR(0)istgenug} we only have to study minimality inside balls centered at zero. This will be the goal of the following subsection.
\subsection{Proof of Minimality}
 Let $B \subset \mathbb{R}^n$ be any ball centered at $0$. We define $B^+:= B \cap \{x_n > 0 \}$, $B^- := B \cap \{x_n < 0 \}.$ Before we start the proof we need several preparatory lemmas. The first one here is a higher order version of an inequality in \cite[Proof of Lemma 2.13]{Velichkov}. We note that \cite[Lemma 2.13]{Velichkov} uses the maximum principle, which is why it is not straightforward to generalize it to our higher order setting.
\begin{lemma}\label{lem:1}
    Let $v \in H^{2}(B)$ be such that $v = 0$ and $\nabla v = 0$ $\mathcal{H}^{n-1}$ a.e. on $\partial B \cap \{ x_n < 0 \}$. Then 
    \begin{equation}\label{eq:3}
        2   \int_{B \cap \{x_n = 0 \}} \partial_n v(x',0) \; \mathrm{d}x'  \leq  \mathcal{F}(v,B^-).
    \end{equation}
    Equality holds if and only if $\Delta v \vert_{B^-} = \chi_{\{v \neq 0 \} \cap B^- }$ a.e..
\end{lemma}
\begin{proof}
Using that $\partial B^- = (B \cap \{x_n = 0 \}) \cup (\partial B \cap \{ x_n < 0 \})$ and $\nabla v = 0$ $\mathcal{H}^{n-1}$ almost everywhere  on  $\partial B \cap \{ x_n < 0 \}$ we find that
    \begin{align*}
          \int_{B \cap \{x_n = 0 \}} \partial_n v(x',0) \; \mathrm{d}x' =  \int_{\partial B^-} (\nabla v \cdot \nu ) \; \mathrm{d}\mathcal{H}^{n-1}.
    \end{align*}
     Applying the Gauss divergence theorem yields
    \begin{align*}
        \int_{B \cap \{x_n = 0 \}} \partial_n v(x',0) \; \mathrm{d}x' = \int_{B^-} \Delta v(x) \; \mathrm{d}x = \int_{\{ v \neq 0 \} \cap B^- } \Delta v(x) \; \mathrm{d}x,
    \end{align*}
    where we used that by \cite[Lemma 4]{GMMathAnn} $\Delta v = 0$ a.e. on $\{ v =  0 \}$. Using that $\Delta v(x) \leq \frac{1 + (\Delta v(x))^2}{2}$ we estimate
    \begin{align}
       &    \int_{B \cap \{x_n = 0 \}} \partial_n v(x',0) \; \mathrm{d}x'   =  \int_{\{ v \neq 0 \} \cap B^- } \Delta v(x) \; \mathrm{d}x   \leq  \int_{\{ v \neq 0 \} \cap B^-} \frac{1 + (\Delta v(x))^2}{2}\; \mathrm{d}x  \label{eq:firstestimate} \\ & = \frac{1}{2} |\{ v \neq 0 \} \cap B^- | +  \frac{1}{2} \int_{\{v \neq 0 \} \cap B^- } |\Delta v(x)|^2 \; \mathrm{d}x \leq \frac{1}{2} \mathcal{F}(v,B^-). \nonumber 
    \end{align}
    Inequality \ref{eq:3} follows. It remains to study the equality case. Notice that the estimate $\Delta v(x) \leq \frac{1 + (\Delta v(x))^2}{2}$ is satisfied with equality if and only if $\Delta v(x) = 1$. Equality in the estimate \eqref{eq:firstestimate} would therefore imply that $\Delta v = 1$ a.e. on $\{ v \neq 0 \} \cap B^-$. Since (again by \cite[Lemma 4]{GMMathAnn}) $\Delta v = 0$ a.e. on $\{ v = 0 \}$ we obtain that $\Delta v \vert_{B^-} = \chi_{B^- \cap \{ v \neq 0 \}}.$ 
\end{proof}

\begin{lemma}
\label{lem:uminimmalinball}    Let $B \subset \mathbb{R}^n$ be a ball centered at zero and  $u \in H^2_{loc}(\mathbb{R}^n)$ be defined as in \eqref{eq:2}. Then for all $v \in H^2(B)$ such that $v- u \in H_0^2(B)$ one has $\mathcal{F}(u,B) \leq \mathcal{F}(v,B)$. Equality holds iff $\Delta v = \chi_{\{v \neq 0 \} }$ a.e. on $B$. 
\end{lemma}
\begin{proof}
    Let $v \in H^2(B)$ be such that $v- u \in H_0^2(B)$. 
    Using that $u \equiv 0$ on $B^-$ we find that
    \begin{align}
        \mathcal{F}(v,B) - \mathcal{F}(u,B) =  \mathcal{F}(v,B^+) + \mathcal{F}(v,B^-) - \mathcal{F}(u,B^+). \label{eq:6}
    \end{align}
    Now we compute 
    \begin{align}
        \mathcal{F}(v,B^+) - \mathcal{F}(u,B^+) & = \int_{B^+} (\Delta v)^2 \; \mathrm{d}x - \int_{B^+} (\Delta u)^2 \; \mathrm{d}x + |\{v \neq 0 \} \cap B^+| - |B^+|  \nonumber
        \\ & = \int_{B^+} (\Delta v)^2 \; \mathrm{d}x - \int_{B^+} (\Delta u)^2 \; \mathrm{d}x - |\{ v = 0 \} \cap B^+|  \nonumber
        \\ & =\int_{B^+} [(\Delta (v -u) )^2 + 2 \Delta u \Delta v - (\Delta u)^2] \; \mathrm{d}x - \int_{B^+} (\Delta u)^2 \; \mathrm{d}x - |\{ v = 0 \} \cap B^+| \nonumber
        \\ & = \int_{B^+}(\Delta (v -u) )^2  \; \mathrm{d}x + 2 \int_{B^+} \Delta u \Delta v \; \mathrm{d}x - 2 \int_{B^+} (\Delta u)^2 \; \mathrm{d}x  - |\{ v  = 0 \} \cap B^+|  \nonumber
        \\ &= \int_{B^+}(\Delta (v -u) )^2  \; \mathrm{d}x +  2 \int_{B^+} \Delta u \Delta (v-u)  \; \mathrm{d}x - | \{ v = 0 \} \cap B^+|. \label{eq:11}
    \end{align}
    Using that $\Delta u = 1$ in $B^+$  and $\nabla u = 0$ on $\{x_n = 0 \}$ yields 
    \begin{align}
         2 \int_{B^+} \Delta u \Delta (v-u)  \; \mathrm{d}x  & = 2\int_{B^+} \Delta (v-u) = 2\int_{\partial B^+} (\nabla (v-u) \cdot \nu) \; \mathrm{d}\mathcal{H}^{n-1} \nonumber \\ &  = 2\int_{B \cap \{ x_n = 0 \}} (-\partial_n) (v-u) \; \mathrm{d}\mathcal{H}^{n-1} = - 2 \int_{B \cap \{ x_n = 0 \}} \partial_n v(x',0) \; \mathrm{d}x'  \label{eq:13} 
    \end{align}
    Furthermore, 
    exploiting again \cite[Lemma 4]{GMMathAnn} and making use of $\Delta u = 1$ on $B^+$, we obtain:
    \begin{align}
        \int_{B^+} (\Delta (v-u))^2 \; \mathrm{d}x & = \int_{B^+ \cap \{ v \neq 0 \} } (\Delta (v - u))^2 \; \mathrm{d}x + \int_{B^+ \cap \{ v =0 \} } (\Delta (v - u))^2 \; \mathrm{d}x \nonumber
       \\ &  =  \int_{B^+ \cap \{ v \neq 0 \} } (\Delta (v - u))^2 \; \mathrm{d}x  + \int_{B^+ \cap \{ v =0 \} } (\Delta u )^2 \; \mathrm{d}x \nonumber
       \\ & =  \int_{B^+ \cap \{ v \neq 0 \} } (\Delta (v - u))^2 \; \mathrm{d}x  + \int_{B^+ \cap \{ v =0 \} } 1 \; \mathrm{d}x \nonumber
       \\ &=  \int_{B^+ \cap \{ v \neq 0 \} } (\Delta (v - u))^2 \; \mathrm{d}x  + |B^+ \cap \{ v = 0 \}|.  \label{eq:16}
    \end{align}
    With the observations of \eqref{eq:13} and \eqref{eq:16} in \eqref{eq:11} we find 
    \begin{align*}
        \mathcal{F}(v,B^+) - \mathcal{F}(u,B^+) = \int_{B^+\cap \{ v \neq 0 \}} (\Delta (v-u))^2 \; \mathrm{d}x  - 2 \int_{B \cap \{x_n = 0 \}} \partial_n v(x',0) \; \mathrm{d}x'. 
    \end{align*}
    Plugging this in \eqref{eq:6} we obtain
    \begin{equation*}
        \mathcal{F}(v,B) - \mathcal{F}(u,B)= \int_{B^+\cap \{ v \neq 0 \}} (\Delta (v-u))^2 \; \mathrm{d}x  - 2 \int_{B \cap \{x_n = 0 \}} \partial_n v(x',0) \; \mathrm{d}x' + \mathcal{F}(v,B^-). 
    \end{equation*}
    Lemma \ref{lem:1} implies that $\mathcal{F}(v,B^-) \geq 2  \int_{B \cap \{x_n = 0 \}} \partial_n v(x',0) \; \mathrm{d}x' $. Therefore,
    \begin{equation}\label{eq:14}
        \mathcal{F}(v,B) - \mathcal{F}(u,B) \geq  \int_{B^+ \cap \{ v \neq 0 \} } (\Delta (v- u))^2 \; \mathrm{d}x \geq 0.
    \end{equation}
    The minimality of $u$ is shown. Finally we turn to the equality case. If \eqref{eq:14} holds with equality it follows that $\Delta (v-u) =0$ a.e. on $B^+ \cap \{ v \neq 0 \}$. Since $\Delta u = 1$ on $B^+$ this implies $\Delta v = 1$ a.e. on $B^+ \cap \{ v \neq 0 \}$. Also recall that we have used the estimate from Lemma \ref{lem:1}. If this estimate holds with equality, Lemma \ref{lem:1} yields that $\Delta v = 1$ a.e. on $B^- \cap \{ v \neq 0 \}$. We conclude that $\Delta v = 1$ a.e. on $\{ v \neq 0 \}$. The claim follows.
\end{proof}

\begin{proof}[Proof of Theorem \ref{thm:Main}] Follows from Lemma \ref{lem:uminimmalinball}  and Lemma \ref{lem:BR(0)istgenug} as a direct consequence.  \end{proof}

We can actually improve Lemma~\ref{lem:uminimmalinball}: Our function $u$ given by \ref{eq:2} is actually \emph{the only} minimizer with its boundary conditions on $\partial B$. This is remarkable since minimizers of Alt-Caffarelli problems are in general not unique, see e.g. \cite[Remark 2]{GMMathAnn}. 

\begin{lemma}\label{lem:aSerrinTypeprobl}
    Let $u \in H^2_{loc}(\mathbb{R}^n)$ be defined as in \eqref{eq:2} and suppose that $v \in H^2(B)$ satisfies 
    \begin{equation*}
        \begin{cases}
            \Delta v = \chi_{\{ v \neq 0 \} }  & \textrm{a.e. in $B$}, \\ v-u \in H_0^2(B).
        \end{cases}
    \end{equation*}
    Then $v = u$ in $B$. 
\end{lemma}
\begin{proof}
For almost every $x \in B^+$ we have
\begin{equation*}
    \Delta (v-u )(x) = \Delta v(x) - \Delta u(x) = \chi_{\{ v\neq 0 \}}(x) - 1 \leq 0.
\end{equation*}
Next compute for almost every $x \in B^-$ 
\begin{equation*}
    \Delta (v-u )(x) = \Delta v(x) - \Delta u(x) = \chi_{\{ v\neq 0 \}}(x) - 0 \geq 0.
\end{equation*}
 As a consequence, $\Delta (v-u) \leq 0$ on $B^+$ and $\Delta (v-u) \geq 0$ on $B^-$.
 This implies that $\Delta (v-u)(x) x_n \leq 0$ for a.e. $x \in B$. Now we integrate this quantity on $B$. This gives taking advantage of $v-u\in H^2_0 (B):$ 
\begin{align*}
     0  & \geq \int_B \Delta (v-u)(x) x_n \; \mathrm{d}x = - \int_B \nabla (v-u) \nabla x_n \; \mathrm{d}x 
      = - \int_B \partial_n (v-u) \; \mathrm{d}x =0.
\end{align*}
We have obtained that $x \mapsto \Delta (v-u)(x) x_n$ is a function on $B$ that is almost everywhere nonpositive and whose integral is zero. As a consequence it   vanishes almost everywhere. We infer that $\Delta v = \Delta u$ almost everywhere on $B \setminus \{ x_n = 0 \}$. Since $\{x_n =0 \}$ is a null set we have that $\Delta v = \Delta u = \chi_{B^+}$ almost everywhere.  Now recall that the solution $w \in H^1(B)$ of 
\begin{equation}\label{eq:Poisson}
    \begin{cases}
        \Delta w = \chi_{B^+}  & \textrm{a.e. in $B$ } \\
        w - u \in H_0^1(B)
    \end{cases}
\end{equation}
is unique. Observe also  that $u$ as well as $v$ satisfy \eqref{eq:Poisson}. This yields $v = u$ in $B$ as claimed. \end{proof}

Combining Lemmas~\ref{lem:uminimmalinball} and \ref{lem:aSerrinTypeprobl}
yields the following result.  
\begin{corollary}\label{ref:cor1}
    Let $u\in H^2_{loc}(\mathbb{R}^n)$ be defined as in \eqref{eq:2} and  $B \subset \mathbb{R}^n$ be a ball centered at zero. Then for all $v \in H^2(B)$ such that $v- u \in H_0^2(B)$ one has $\mathcal{F}(u,B) \leq \mathcal{F}(v,B)$. Equality holds if and only if $v = u$ in $B$. 
\end{corollary}

\begin{remark}
    Let $\Omega \subset \mathbb{R}^n$ be an open bounded set. We have shown that for all $v \in H^2(\Omega)$ such that $v-u \in H_0^2(\Omega)$ one has $\mathcal{F}(v,\Omega) \geq \mathcal{F}(u,\Omega)$. We claim that equality holds iff $v \equiv u$ on $\Omega$.  If $\Omega$ is a ball centered at zero this is Corollary \ref{ref:cor1}. Making use once more of Lemma \ref{lem:3}
    yields the result for general $\Omega$.
\end{remark}

\subsection{Nonminimality for the one-sided functional}
In literature the following \emph{one-sided} Alt-Caffarelli functional has also been frequently studied
\begin{equation*}
    \tilde{\mathcal{F}}(u,\Omega) := \int_\Omega (\Delta u)^2 \; \mathrm{d}x + |\{ x \in \Omega: u(x) > 0 \}|  ,
\end{equation*}
see e.g. \cite{DKV19,DKV20}.
One fundamental difference between $\mathcal{F}$ and $\tilde{\mathcal{F}}$ concerns the expected behaviour of the gradient of minimizers $u$ on the free boundary. 
As was  shown in \cite[Theorem 1.4]{MuellerAMPA} for $n=2$,  minimizers of $\tilde{\mathcal{F}}(\cdot, \Omega)$ with positive boundary conditions have nonvanishing gradient on their free boundary $\partial \{ u > 0 \}$. For $\mathcal{F}(\cdot, \Omega)$ one expects a different behaviour: As was shown in \cite[Theorem 3]{GMMathAnn} $\{ u = 0 \}$ has  positive Lebesgue measure in sufficiently large domains or for sufficiently small boundary data. 
Since minimizers with positive boundary data lie in $C^1$ (cf.  \cite[Theorem 2]{GMMathAnn}), there will certainly be points on the free boundary with vanishing gradient. See also \cite[Examples 2 and 4]{GMMathAnn} for radial minimizers of $\mathcal{F}$ in $n=2$ with vanishing gradient on $\partial \{ u \neq 0 \}$. Notice also that the half-plane minimizer in \eqref{eq:2} has vanishing gradient on its free boundary  $\partial \{ u \neq 0 \}$. Therefore, it seems reasonable to believe that it is a nonminimal configuration in the one-sided setting. In the case $n= 1$ we can in fact prove this.

\begin{lemma}\label{lem:onesided}
    Let $n = 1$ and  $u \in H^2_{loc}(\mathbb{R})$ be defined as \eqref{eq:2}. Then $u$ is not a global minimizer for $\tilde{\mathcal{F}}$. More precisely, there exists $v \in H^2((-1,1))$ such that $v-u \in H_0^2((-1,1))$ and $\tilde{\mathcal{F}}(v, (-1,1)) < \mathcal{\tilde{F}}(u,(-1,1)).$
\end{lemma}
\begin{proof}
    Define $v: (-1,1) \rightarrow \mathbb{R}, v(x)= \frac{1}{8}x (x+1)^2$. Notice that $v(1) = \frac{1}{8}(1+1)^2= \frac{1}{2} = u(1)$ and $v(-1)=0 = u(-1)$. Moreover, due to 
  $
        v'(x) = \frac{1}{4} x(x+1) + \frac{1}{8}(x+1)^2
   $
    we have that $v'(1)= 1 = u'(1)$ and $v'(-1) =0 = u'(-1)$. This implies that $v-u \in H_0^2((-1,1))$. Now  $v''(x) = \frac{1}{4} (2x+1) + \frac{1}{4} (x+1) = \frac{1}{4} (3x+2)$ and $\{v > 0 \} = (0,1)$ yields
    \begin{equation}\label{eq:27}
        \tilde{\mathcal{F}}(v,(-1,1))  = \int_{(-1,1)} \left(\tfrac{1}{4}(3x+2)\right)^2 \; \mathrm{d}x + |(0,1)| = \frac{7}{8} + 1 < 2. 
    \end{equation}
    However, since $u$ satisfies $u'' =  \chi_{(0,1)}$ and $ \{ u > 0 \} = (0,1)$ we infer
    \begin{equation*}
        \tilde{\mathcal{F}}(u,(-1,1))  = \int_{(-1,1)} (\chi_{(0,1)}(x))^2 \; \mathrm{d}x + |(0,1)| = 2 |(0,1)| = 2. 
    \end{equation*}
    This and \eqref{eq:27} show that $u$ is not  minimal. 
\end{proof}

\section{Type II: Minimizers with constant Laplacian}\label{sec:constLaplacian}

In this section we discuss
a class of global minimizers which are not half-plane solutions and show with this that $d^{**} = 2$ (cf. Definition \ref{eq:defcriticaldimbih}). The following is our main finding in this section.

\begin{theorem}\label{lem:globmin}
    Let $n \geq 1$ and $w \in H^2_{loc}(\mathbb{R}^n)$ be such that $\Delta w \equiv C$ for some $C \in (-\infty,-1] \cup [1, \infty)$. Then $w$ is a global minimizer of $\mathcal{F}$. 
\end{theorem}
\begin{proof} 
If $\Delta w \equiv C$ for some $C \neq 0$ then $|\{ w = 0 \}| = 0$. This is due to the fact that by \cite[Lemma 4]{GMMathAnn} one has $\Delta w = 0$ a.e. on $\{ w = 0 \}.$
By Lemma \ref{lem:BR(0)istgenug} it suffices to show that for all balls $B$ centered in $0$ and $v \in H^2(B)$ such that $v-w \in H_0^2(B)$ one has $\mathcal{F}(v,B) \geq \mathcal{F}(w,B)$. To this end we compute 
    \begin{align}
        \mathcal{F}(v,B)- \mathcal{F}(w,B)  & = \int_B(\Delta v)^2 \; \mathrm{d}x - \int_B (\Delta w)^2 \; \mathrm{d}x + |\{ v \neq 0 \} \cap B|  - |\{ w \neq 0 \} \cap B|  \nonumber
        \\ & = \int_B(\Delta v)^2 \; \mathrm{d}x - \int_B (\Delta w)^2 \; \mathrm{d}x + |\{ v \neq 0 \} \cap B|  - |B|  \nonumber
        \\ & = \int_B(\Delta v)^2 \; \mathrm{d}x - \int_B (\Delta w)^2 \; \mathrm{d}x - |\{ v = 0 \} \cap B|  \nonumber
        \\ & = \int_B [(\Delta v-w)^2 + 2\Delta v \Delta w - (\Delta w)^2] \; \mathrm{d}x - \int_B (\Delta w)^2 \; \mathrm{d}x - |\{ v = 0 \} \cap B| \nonumber
        \\ & = \int_B (\Delta (v-w))^2 \; \mathrm{d}x + 2 \int_B \Delta w \Delta(v-w) \; \mathrm{d}x - |\{ v = 0 \} \cap B|. \label{eq:40A}
     \end{align}
     Now $\Delta w \equiv C$ yields
     \begin{equation*}
         2 \int_B \Delta w \Delta(v-w) \; \mathrm{d}x = 2C \int_B \Delta (v-w) \; \mathrm{d}x = 2 \int_{\partial B} \partial_\nu (v-w) \; \mathrm{d}\mathcal{H}^{n-1} = 0,
     \end{equation*}
     since $v-w \in H_0^2(B).$ Moreover, by \cite[Lemma 4]{GMMathAnn}
     \begin{align*}
         \int_B (\Delta (v-w))^2 \; \mathrm{d}x &  = \int_{B \cap \{ v\neq 0 \}}  (\Delta (v-w))^2 \; \mathrm{d}x  + \int_{B \cap \{ v= 0 \}}  (\Delta (v-w))^2 \; \mathrm{d}x \\ & =  \int_{B \cap \{ v\neq 0 \}}  (\Delta (v-w))^2 \; \mathrm{d}x  + \int_{B \cap \{ v= 0 \}}  (\Delta w)^2 \; \mathrm{d}x  
         \\ &= \int_{B \cap \{ v\neq 0 \}}  (\Delta (v-w))^2 \; \mathrm{d}x  + \int_{B \cap \{ v= 0 \}}  C^2 \; \mathrm{d}x \\ &  = \int_{B \cap \{ v\neq 0 \}}  (\Delta (v-w))^2 \; \mathrm{d}x  + C^2|B \cap \{ v = 0 \} |. 
     \end{align*}
     Using the previous two computations in \eqref{eq:40A} we find 
     \begin{equation*}
         \mathcal{F}(v,B)- \mathcal{F}(w,B) = (C^2-1)|B \cap \{ v = 0 \} | +   \int_{B \cap \{ v \neq 0 \} } (\Delta (v-w))^2 \; \mathrm{d}x \geq 0,
      \end{equation*}
      due to the fact that $C^2 \geq 1$. The claim follows. 
\end{proof}

The above lemma gives rise to infinitely many two-homogeneous global minimizers of $\mathcal{F}$. Examples are $\frac{1}{2} x_1^2, \frac{1}{4}(x_1^2+ x_2^2) , \frac{1}{2}(2x_1^2-x_2^2),...$

\begin{corollary}\label{cor:dstarstar}
   Let $d^{**}$ be as in Definition \ref{eq:defcriticaldimbih}. Then $d^{**}  =2$. 
\end{corollary}
\begin{proof}
    Define $w \in H^2_{loc}(\mathbb{R}^2), w(x)=\frac{1}{4}|x|^2$. According to Theorem \ref{lem:globmin}, $w$ is a global minimizer. Clearly $w$ is two-homogeneous. But, the free boundary $\{ w= 0, \nabla w = 0 \} = \{0 \}$ is not a smooth one-dimensional $C^1$-submanifold of $\mathbb{R}^2$. 
\end{proof}


Again, for the above result $w$ need not be assumed two-homogeneous. 

\begin{remark}
    One could now conjecture that all $w \in H^2_{loc}(\mathbb{R}^n)$ with constant Laplacian yield global minimizers. However,
    this not the case, not even for $n=2$. Indeed, choosing $w \equiv u_0$ for some suitably small constant $u_0 > 0$ (i.e. $\Delta w \equiv 0$), \cite[Example 4]{GMMathAnn} yields that $w$ is not a global minimizer. This is due to the fact that according to \cite[Example 4]{GMMathAnn} one finds some $v \in H^2(B_1(0))$ such that $v-u_0 \in H_0^2(B_1(0))$ and $\mathcal{F}(v,B_1(0)) < \pi = \mathcal{F}(u_0,B_1(0)).$ 
    
\end{remark}
Summarizing, we have seen that solutions $w \in H^2_{loc}(\mathbb{R}^n)$ of $\Delta w \equiv \mathrm{const.}$ are global minimizers if the constant has absolute value larger than one but that they are not all global minimizers if the constant is zero. We can further prove that this conclusion persists as long as the constant $\Delta w $ belongs to a sufficiently small interval around $0$. 
\begin{lemma} \label{lem:lemma8}
    Let $n \in \mathbb{N}$ and  $\varepsilon> 0$.
    We define  $w_{\varepsilon}\in H^2_{loc}(\mathbb{R}^n)$, $w_{\varepsilon}(x) := \varepsilon |x|^2$. Then there exists $\varepsilon_0 = \varepsilon_0(n) > 0$ such that  $w_\varepsilon$ is not a global minimizer for all $\varepsilon \in (0, \varepsilon_0)$.%
    \footnote{Recall that $w_\varepsilon$ is on the other hand  a global minimizer for sufficiently large values of $\varepsilon$, cf. Theorem \ref{lem:globmin}.} 
\end{lemma}
\begin{proof} 
    Fix $\varepsilon > 0$. For $\rho \in (0,1)$ to be chosen later, we look at 
    \begin{equation*}
        v_{\varepsilon,\rho}(x) := \begin{cases}
            0 & |x| \le \rho, \\ \varepsilon \frac{(|x|- \rho)^2}{(1-\rho)^3} \left(-2 \rho |x| + (1+\rho)\right)  & |x|> \rho. 
        \end{cases}
    \end{equation*}
    We claim that $w_\varepsilon - v_{\varepsilon,\rho} \in H_0^2(B_1(0)).$ Indeed, for $x \in \mathbb{R}^n$ such that $|x| = 1$ one has
    \begin{equation*}
        v_{\varepsilon,\rho}(x)= \varepsilon\tfrac{(1-\rho)^2}{(1-\rho)^3} ( 1+ \rho- 2\rho) =  \varepsilon = w_\varepsilon(x)
    \end{equation*}
    and 
    \begin{align*}
       \nabla v_{\varepsilon,\rho}(x) & =2 \varepsilon  \tfrac{(|x|- \rho)}{(1-\rho)^3} \left(-2 \rho |x| + (1+\rho)\right) \tfrac{x}{|x|}  -2\rho \varepsilon \tfrac{(|x|-\rho)^2}{(1-\rho)^3}  \tfrac{x}{|x|} = \left( 2\varepsilon \tfrac{1}{1-\rho} - 2\rho \varepsilon \tfrac{1}{1-\rho} \right)x 
       \\ & = 2\varepsilon x = \nabla w_{\varepsilon}(x). 
    \end{align*}
    Using now that $\Delta |x| = \frac{n-1}{|x|}$ and $$\Delta (|x|- \rho)^2 = \sum_{i = 1}^n \partial_i (2 (|x|- \rho) \tfrac{x_i}{|x|} ) = \sum_{i = 1}^n 2 \tfrac{x_i^2}{|x|^2} + 2 (|x|-\rho) (\tfrac{1}{|x|}- \tfrac{x_i^2}{|x|^3}) = 2 + 2\tfrac{(n-1)(|x|- \rho)}{|x|}$$ we obtain for $\rho < |x|< 1$
    \begin{align*}
        \Delta v_{\varepsilon, \rho}(x) &= \tfrac{\varepsilon}{(1-\rho)^3} [(-2\rho |x| + (1+\rho))\Delta (|x|-\rho)^2 \\ & \qquad \qquad \qquad \quad + 2\nabla(|x|-\rho)^2 \cdot \nabla (-2\rho |x| + 1+ \rho) + (|x|-\rho)^2 \Delta (-2\rho |x|+ (1+\rho) ]
        \\ & = \tfrac{\varepsilon}{(1-\rho)^3} \left( (-2\rho |x| + 1+ \rho) \left( 2 + 2 \tfrac{(n-1)(|x|-\rho)}{|x|}\right)  + \sum_{i = 1}^n 4 (|x|- \rho) \tfrac{x_i}{|x|} (-2\rho) \tfrac{x_i}{|x|} + (|x|- \rho)^2 \tfrac{-2\rho(n-1)}{|x|}  \right) 
        \\ &= \tfrac{\varepsilon}{(1-\rho)^3} \left( (-2\rho |x| + 1+ \rho) \left( 2 + 2 \tfrac{(n-1)(|x|-\rho)}{|x|}\right)  - 8 \rho (|x|- \rho)  -2\rho(n-1) \tfrac{(|x|- \rho)^2}{|x|} \right) 
    \end{align*}
    Using that $\frac{|x|- \rho}{|x|}< 1$ we can estimate 
    \begin{equation*}
        |\Delta v_{\varepsilon,\rho}(x)| \leq \tfrac{\varepsilon}{(1-\rho)^3} ( (1+\rho)(2 + 2(n-1)) + 8 \rho  + 2 \rho(n-1)) = \tfrac{\varepsilon}{(1-\rho)^3} ( 2n(1+\rho) + 2\rho(n+3)) = \tfrac{\varepsilon}{(1-\rho)^3} ((4n+6)\rho +2n)
    \end{equation*}
    In particular, 
    \begin{align*}
        \mathcal{F}(v_{\varepsilon,\rho},B_1(0))  & \leq \tfrac{\varepsilon^2}{(1-\rho)^6} ((4n+6)\rho +2n)^2 (|B_1(0)| - |B_\rho(0)| ) + (|B_1(0)| - |B_\rho(0)| )  \\ & \leq |B_1(0)|  \left( \tfrac{\varepsilon^2}{(1-\rho)^6} ((4n+6)\rho +2n)^2 +   (1-\rho^n) \right) 
    \end{align*}
    Note that 
    \begin{equation*}
        \mathcal{F}(w_\varepsilon,B_1(0)) = (2n)^2\varepsilon^2 |B_1(0)|  + |B_1(0)|= |B_1(0)| [\varepsilon^2(2n)^2 + 1]
    \end{equation*}
    Therefore 
    \begin{align*}
       \tfrac{1}{|B_1(0)|} [\mathcal{F}(v_{\varepsilon,\rho},B_1(0)) - \mathcal{F}(w_\varepsilon,B_1(0)) ] & \leq \varepsilon^2 \left( \frac{1}{(1-\rho)^6} ((4n+6)\rho +2n)^2 - (2n)^2 \right) - \rho^n. 
    \end{align*}
    Now assuming $\rho < \frac{1}{2}$ we obtain 
    \begin{equation*}
         \tfrac{1}{|B_1(0)|} [\mathcal{F}(v_{\varepsilon,\rho},B_1(0)) - \mathcal{F}(w_\varepsilon,B_1(0)) ] \leq \varepsilon^2 (2^6 (4n+3)^2- (2n)^2) - \rho^n. 
    \end{equation*}
    Now, if one can find $\rho \in (0,\frac{1}{2})$  such that $\rho^n > \varepsilon^2 (2^6 (4n+3)^2- (2n)^2)$ then the expression above becomes smaller than zero and therefore $w_\varepsilon$ is not a global minimizer. Such parameter $\rho < \frac{1}{2}$ can be chosen in the desired way if and only if
    \begin{equation}
        (\tfrac{1}{2})^n  > \varepsilon^2 (2^6 (4n+3)^2- (2n)^2) , \qquad \textrm{that is} \qquad \varepsilon < \frac{1}{\sqrt{2^n(2^6 (4n+3)^2- (2n)^2)}} =: \varepsilon_0(n). 
    \end{equation}
    \end{proof}

The previous result suggests that there exists some $\lambda_{\operatorname{crit}} \in (0,1]$ (possibly dependent on $n$) such that 
\begin{equation*}
\begin{cases}
\Delta w \equiv \lambda \; \textrm{for some $|\lambda| \geq \lambda_{\operatorname{crit}}$ implies that $w$ is a global minimizer,} \\ \textrm{for each $|\lambda| < \lambda_{\operatorname{crit}}$ there are solutions of $\Delta w \equiv \lambda$ which are not global minimizers.} 
\end{cases}
\end{equation*}
In case that this can be shown, the proof above gives a lower bound on $\lambda_{\operatorname{crit}}$, namely $\lambda_{\operatorname{crit}} \geq \frac{1}{\sqrt{2^n(2^6 (4n+3)^2- (2n)^2)}}$.




\section{Nonextendability of categories (I) and (II)}\label{sec:wd1d2}

We have now studied two types of minimizers. It is natural to ask the question whether there are more minimizers that have a \emph{similar form} to these two types. A natural question to ask is e.g. whether \emph{linear combinations} of global minimizers yield global minimizers again. This is not true for Type (II) minimizers. Indeed, Lemma \ref{lem:lemma8} shows that multiplying the Type-(II) global minimizer $w(x)= |x|^2$ by a small constant yields a function that is not a global minimizer anymore. We can however also study linear combinations of Type-(I) minimizers.
For example, two global minimizers of Type (I) are given by $u_1(x) := \frac{1}{2} x_n^2 \chi_{(0,\infty)}(x_n)$ and $u_2(x) = \frac{1}{2}(-x_n)^2 \chi_{(0,\infty)}(-x_n)= \frac{1}{2}x_n^2 \chi_{(-\infty,0)}(x_n)$ (which is a global minimizer as an orthogonal transformation of $u_1$, cf. also Lemma \ref{lem:invariance}). The following theorem examines global minimality properties of linear combinations of the form $d_1u_1+d_2u_2$, $d_1,d_2 \in \mathbb{R}$ and finds that Type (I) is very rigid with respect to linear combinations. This justifies that Category (I) can not be extended.  
\begin{theorem} \label{theorem:classification_global_min}
    For $d_1,d_2 \in \mathbb{R}$ define 
    \begin{equation}\label{eq:18}
        w_{d_1,d_2} : \mathbb{R}^n \rightarrow \mathbb{R}, \qquad w_{d_1,d_2}(x) := \begin{cases}
            d_1 \frac{x_n^2}{2},\quad & x_n \geq  0, \\ d_2 \frac{x_n^2}{2}, & x_n < 0. 
        \end{cases}  
    \end{equation}
    If $w_{d_1,d_2}$ is a nontrivial global minimizer, then one out of the following two cases occurs: 
    \begin{itemize}
        \item[$\mathrm{(I)}$] $d_1 = \pm 1, d_2 = 0$ or $d_1= 0,d_2 = \pm 1$, i.e. $w_{d_1,d_2}$ is a  Type-$\mathrm{(I)}$-minimizer.
        \item[$\mathrm{(II)}$] $d_1,d_2 \neq 0$ and $d_1 = d_2$, i.e. $w_{d_1,d_2}$ is a Type-$\mathrm{(II)}$-minimizer.
    \end{itemize}
\end{theorem}

\begin{proof}Of course, $w_{0,0}(x)\equiv 0$ is a global minimizer but is not to be considered by assumption.
	So, we have to distinguish the cases whether none or one  of the $d_j$ is equal to $0$.
	\begin{itemize}
		\item[(a)] \textit{Case $d_1\not=0$, $d_2=0$.} 
		Let $B \subset \mathbb{R}^n$ be a ball centered at $0$ and let $\psi \in C_0^\infty(B;\mathbb{R}^n)$ be arbitrary. Since $w_{d_1,d_2}$ is assumed to be a global minimizer we take from \cite[Lemma 4.4]{DKV20} by  obviously changing $\chi_{\{u>0\}}$ into $\chi_{\{u\not=0\}}$:
		\begin{align}
			0  = &- \int_{B\cap\{ w_{d_1,d_2} \neq 0 \}} \mathrm{div}(\psi) \; \mathrm{d}x - \int_B (\Delta w_{d_1,d_2})^2 \; \mathrm{div}(\psi) \; \mathrm{d}x\nonumber\\
			& + 2 \int_B \Delta w_{d_1,d_2} \left( \sum_{i = 1}^n \left(2 \partial_i \psi \cdot \partial_i(\nabla w_{d_1,d_2}) + \partial^2_{ii} \psi \cdot \nabla  w_{d_1,d_2} \right) \right)\; \mathrm{d}x.\label{eq:27A}
		\end{align}
	Here ``$\, \cdot\, $'' denotes the canonical scalar 
    product in $\mathbb{R}^n$.
		We compute the terms in \eqref{eq:27A} and find
		\begin{equation*}
			\int_{B\cap \{ w_{d_1,d_2} \neq 0 \}} \mathrm{div}(\psi) \; \mathrm{d}x = \int_{B^+} \mathrm{div}(\psi) \; \mathrm{d}x = \int_{\partial B^+} \psi \cdot \nu \; \mathrm{d}\mathcal{H}^{1} = - \int_{B \cap \{ x_n = 0 \}} \psi_n \; \mathrm{d}\mathcal{H}^{n-1}  
		\end{equation*}
		and 
		\begin{equation*}
			\int_B (\Delta w_{d_1,d_2})^2 \; \mathrm{div}(\psi) \; \mathrm{d}x = d_1^2 \int_{B^+} \mathrm{div}(\psi) \; \mathrm{d}x = - d_1^2  \int_{B \cap \{ x_n = 0 \}} \psi_n \; \mathrm{d}\mathcal{H}^1.  
		\end{equation*}
		For the third term notice that 
		\begin{align*}
			\sum_{i = 1}^n (2 \partial_i \psi \cdot \partial_i(\nabla w_{d_1,d_2}) +& \partial^2_{ii} \psi \cdot \nabla  w_{d_1,d_2} )\\
			 =&
			 \sum_{i = 1}^n \partial_i \psi \cdot  \partial_i (\nabla w_{d_1,d_2}) + [\partial_i \psi \cdot \partial_i(\nabla w_{d_1,d_2}) + \partial^2_{ii} \psi \cdot \nabla  w_{d_1,d_2} ]\\
			 =& \sum_{i = 1}^n \partial_i \psi \cdot  \partial_i (\nabla w_{d_1,d_2}) +  \sum_{i = 1}^n \partial_i (\partial_i \psi \cdot \nabla w_{d_1,d_2}). 
		\end{align*}
		This decomposes the third term into two further summands. For the last summand we compute
		\begin{align*}
			2 \int_B \Delta w_{d_1,d_2} \left( \sum_{i = 1}^n \partial_i (\partial_i \psi \cdot \nabla w_{d_1,d_2}) \right) \; \mathrm{d}x =& 2d_1 \sum_{i = 1}^n \int_{B^+} \partial_i(\partial_i \psi \cdot \nabla w_{d_1,d_2}) \; \mathrm{d}x\\
			= &2d_1 \sum_{i = 1}^n \int_{\partial B^+} (\partial_i \psi \cdot \nabla w_{d_1,d_2}) \nu_i \; \mathrm{d}\mathcal{H}^{n-1} = 0,
		\end{align*}
		where we have used that for $i= 1,..., n$ one has $\partial_i \psi = 0$ on $\partial B$ and $\nabla w_{d_1,d_2}= 0$ on $B \cap \{x_n = 0 \}$. Moreover, for the first (additional) summand we have 
		\begin{align}
			 2 \int_B \Delta w_{d_1,d_2} &\left( \sum_{i = 1}^n \partial_i \psi \cdot \partial_i  \nabla w_{d_1,d_2} \right) \; \mathrm{d}x   
			 = 2d_1 \sum_{i = 1}^n \int_{B^+} \partial_i \psi \cdot \partial_i \nabla w_{d_1,d_2} \; \mathrm{d}x\nonumber\\
			=&2 d_1^2 \sum_{i = 1}^n \int_{B^+} \partial_i \psi \cdot \partial_i \left(x_n e_n\right)\; \mathrm{d}x
			  =2 d_1^2 \int_{B^+} \partial_n\psi_n \; \mathrm{d}x \nonumber\\
			=& - 2 d_1^2 \int_{B \cap \{x_n = 0 \}} \psi_n \; \mathrm{d}\mathcal{H}^{n-1} . \label{eq:differentsign}
		\end{align}
		Using all the previous computations in \eqref{eq:27A} we obtain 
		\begin{equation*}
			0 = (1+ d_1^2 - 2d_1^2) \int_{B \cap \{x_n = 0 \} } \psi_n \; \mathrm{d}\mathcal{H}^{n-1} = (1-d_1^2)  \int_{B \cap \{x_n = 0\} } \psi_n \; \mathrm{d}\mathcal{H}^{n-1}.  
		\end{equation*}
		Choosing a test vector field $\psi\in C_0^\infty(B;\mathbb{R}^n)$ such that $\int_{B \cap \{x_n = 0 \} } \psi_n \; \mathrm{d}\mathcal{H}^{n-1} \neq 0$ we find $d_1^2 = 1$, that is $d_1 = \pm 1$.

		\item[(b)]\textit{Case $d_1=0$, $d_2\not=0$.} 
		 Applying Lemma~\ref{lem:invariance} to the previous case yields that $d_2=\pm 1$.
		\item[(c)]\textit{Case $d_1\not=0$, $d_2\not=0$.}
		Let $B \subset \mathbb{R}^n$ be a ball centered at $0$ and let $\varphi \in C_0^\infty(B;\mathbb{R})$ be arbitrary.
		Then for $t\in\mathbb{R}\setminus \{0\}$ (but close to $0$) we have
		\begin{align*}
			{\mathcal F} ( w_{d_1,d_2}) =&\int_B (\Delta w_{d_1,d_2})^2\, \mathrm{d}x +|B|\\
			\le & 	{\mathcal F} ( w_{d_1,d_2}+t\varphi ) =\int_B (\Delta w_{d_1,d_2}+t\varphi)^2\, \mathrm{d}x +|\{x\in B:\, w_{d_1,d_2}+t\varphi\not=0\}|\\
			\le & \int_B (\Delta w_{d_1,d_2}+t\varphi)^2\, \mathrm{d}x + |B|\\
			=&\int_B (\Delta w_{d_1,d_2})^2\, \mathrm{d}x +2t \int_B \Delta w_{d_1,d_2}\cdot \Delta \varphi \, \mathrm{d}x
			+ t^2 \int_B (\Delta \varphi)^2\, \mathrm{d}x+|B|\\
			\Rightarrow \quad 0\le & 2t \int_B \Delta w_{d_1,d_2}\cdot \Delta \varphi \, \mathrm{d}x
			+ t^2 \int_B (\Delta \varphi)^2\, \mathrm{d}x.
		\end{align*}
	Since the sign of $t$ may be arbitrary this gives
	$$
	\forall \varphi \in C_0^\infty(B;\mathbb{R}):\quad \int_B \Delta w_{d_1,d_2}\cdot \Delta \varphi \, \mathrm{d}x=0.
	$$
	This means that $w_{d_1,d_2}$ is weakly biharmonic and hence, by elliptic regularity, strongly biharmonic in $B$. In particular, $w_{d_1,d_2} \in C^\infty(B)$. This implies that $d_1=d_2$. 
	\end{itemize}
	
\end{proof}

\begin{remark}
  The proof of the previous theorem is based on \eqref{eq:27A}, taken from \cite[Lemma 4.4]{DKV20}. In \cite{DKV20} this was derived by means of \emph{inner variation} methods. It is remarkable that the behaviour  of the Lebesgue measure term under inner variation can be understood much better compared to the classical perturbation approach.
\end{remark}

\section{No measure-valued equation for minimizers}\label{sec:NoMeasValuedInequality}
In one-sided biharmonic free boundary problems such as the one-sided Alt-Caffrelli problem \cite{MuellerAMPA} or the biharmonic obstacle problem \cite{CaffarelliFriedman}, minimizers $v \in H^2(\Omega)$ satisfy a measure-valued PDE of the form 
\begin{equation*}
    \int_{\Omega} \Delta v \Delta \varphi \; \mathrm{d}x = \int_{\Omega} \varphi \; \mathrm{d}\mu \qquad \textrm{for all $\varphi \in C_0^\infty(\Omega)$},
\end{equation*}
where $\mu$ is a Radon measure. The reason for this is that imposing a sign constraint on admissible perturbations  gives rise to a (distributional) \emph{differential inequality.} 
The well-known Riesz-Markov-Kakutani representation theorem \cite[Theorem 1.39]{EvansGariepy} converts such inequalities to a measure-valued PDE as the one stated above. 

For two-sided problems it seems reasonable to expect a signed (!) Radon measure $\mu$ such that 
\begin{equation}\label{eq:measvaleq}
    \int_{\Omega} \Delta v \Delta \varphi \; \mathrm{d}x = \int_{\Omega} \varphi \; \mathrm{d}\mu \qquad \textrm{for all $\varphi \in C_0^\infty(\Omega)$}.
\end{equation}
Such equation is for example satisfied in the biharmonic two-obstacle problem \cite{TwoSidedObstacle}. We will show in the following that (surprisingly) for the two-sided Alt-Caffarelli problem, such a measure-valued equation may in general not 
hold. We will actually present two reasons for this. 
\begin{itemize}
    \item[(i)] We will investigate the properties of explicit minimizers in 1-d found in \cite{GMMathAnn}. These are explicitly computed (under Navier boundary conditions). We will characterize the distribution $\varphi \mapsto \int_\Omega u''(x) \varphi''(x) \; \mathrm{d}x$ for any such explicit minimizer $u$ and show that this distribution can never be given by a measure. 
    \item[(ii)] We will argue that solutions of \eqref{eq:measvaleq} must necessarily have $W^{3,1}_{loc}$-Sobolev regularity. However, the explicit radial minimizers in \cite[Example 2, Example 4]{GMMathAnn} and the half-space minimizers found in Theorem \ref{thm:Main} do not enjoy $W^{3,1}_{loc}$-regularity. The regularity observation is independent of the boundary conditions and the explicit minimizers whose regularity we investigate satisfy both Dirichlet and Navier boundary conditions.  
 \end{itemize}

\subsection{Distributional equations for explicit solutions}
In \cite[Example 1]{GMMathAnn} it has been found that for $\Omega= (-R,R)$, $R > \sqrt{3} + \frac{1}{\sqrt{3}}$ and $u_0 \equiv 1$ the minimizer of $\mathcal{F}$ in 
$
    \mathscr{N} := \{ u \in H^2(\Omega) : u - u_0 \in H_0^1(\Omega) \}
$
is uniquely determined by 
\begin{equation*}
    u(x) = \begin{cases}
        0 & |x| \leq R-\sqrt{3}, \\ 
        \frac{(|x|-R+ \sqrt{3})^2(R+2\sqrt{3}-|x|) }{6\sqrt{3}} & |x|> R- \sqrt{3}.
    \end{cases}
\end{equation*}
We compute the derivatives of $u$ and obtain 
\begin{align*}
     u'(x)  & = \begin{cases}
        0 & |x| \leq R-\sqrt{3} \\ \frac{1}{6\sqrt{3}} [ 2 (|x|- R+ \sqrt{3}) (R + 2 \sqrt{3}-|x|) \mathrm{sgn}(x) + (|x|-R+ \sqrt{3})^2 (-\mathrm{sgn}(x)) ] & |x|> R- \sqrt{3}
    \end{cases},\\
     u''(x) & = \begin{cases}
        0 & |x| \leq R-\sqrt{3} \\ \frac{1}{6\sqrt{3}} [ 2  (R + 2 \sqrt{3}-|x|) - 4 (|x|-R+ \sqrt{3})  ] & |x|> R- \sqrt{3}
    \end{cases}
     = \begin{cases}
        0 & |x| \leq R-\sqrt{3} \\ \frac{1}{\sqrt{3}} [  R - |x| ] & |x|> R- \sqrt{3}
    \end{cases}.
\end{align*}
Now define $h_1,h_2: \Omega \rightarrow \mathbb{R}$ 
\begin{equation*}
    h_1(x) := R \chi_{(R-\sqrt{3},R)}(|x|), \quad h_2(x) := |x| \chi_{(R-\sqrt{3},R)}(|x|).
\end{equation*}
For any $\varphi \in C_0^\infty(\Omega)$ we  compute further
\begin{equation*}
    \int_\Omega h_1(x) \varphi''(x) \; \mathrm{d}x  = R\varphi'(-R+ \sqrt{3}) - R\varphi'(R - \sqrt{3})
\end{equation*}
and 
\begin{align*}
   & \int_\Omega h_2(x) \varphi''(x) \; \mathrm{d}x  = \int_{-R}^{-R+\sqrt{3}} -x\varphi''(x) \; \mathrm{d}x + \int_{R- \sqrt{3}}^R x\varphi''(x) \; \mathrm{d}x \\&  = -(-R+\sqrt{3})\varphi'(-R+ \sqrt{3}) - (R-\sqrt{3}) \varphi'(R - \sqrt{3}) + \int_{-R}^{-R+\sqrt{3}} \varphi'(x) \; \mathrm{d}x - \int_{R-\sqrt{3}}^R \varphi'(x) \; \mathrm{d}x
   \\ & = -(-R+\sqrt{3})\varphi'(-R+ \sqrt{3}) - (R-\sqrt{3}) \varphi'(R - \sqrt{3}) + \varphi(-R + \sqrt{3}) + \varphi(R-\sqrt{3}). 
    \end{align*}
Using the formula for $u''$ above we find 
\begin{align*}
    \int_\Omega u''(x) \varphi''(x) \; \mathrm{d}x  = (\varphi'(-R+\sqrt{3}) - \varphi'(R- \sqrt{3}) ) - \frac{1}{\sqrt{3}} (\varphi(R-\sqrt{3}) + \varphi(-R+\sqrt{3})).  
\end{align*}
The distribution $ \varphi \mapsto \frac{1}{\sqrt{3}} (\varphi(R-\sqrt{3}) + \varphi(-R+\sqrt{3}))$ can clearly be represented by a Radon measure on $\Omega$ since for each $\varphi \in C_0^\infty(\Omega)$ one has
\begin{equation*}
    \varphi(R-\sqrt{3}) + \varphi(-R+\sqrt{3}) = \int_\Omega \varphi \; \mathrm{d}(\delta_{R-\sqrt{3}} + \delta_{-R+\sqrt{3}}).
\end{equation*}
However, the distibution $\varphi \mapsto \varphi'(-R+\sqrt{3}) - \varphi'(R- \sqrt{3})$ is not represented by a Radon measure as the following lemma shows. In view of this the distribtion
\begin{equation*}
    u^{(4)} : \varphi \mapsto \int_\Omega u''(x) \varphi''(x) \; \mathrm{d}x
\end{equation*}
cannot be represented by a Radon measure. 
\begin{lemma}
    Let $\Omega = (-R,R)$ for any $R> 0$, $m \in \mathbb{N}$, $a_1,\ldots,a_m \in \Omega$ mutually distinct 
    and $c_1,\ldots,c_m \in \mathbb{R}$. If there exists a signed Radon measure $\mu$ on $\Omega$ such that the distribution $\varphi \mapsto \sum_{i = 1}^m c_i \varphi'(a_i)$ is represented by $\mu$, that is
    \begin{equation}\label{eq:Ableitungalsmass}
        \sum_{i = 1}^m c_i \varphi'(a_i) = \int_\Omega \varphi \; \mathrm{d}\mu \qquad \textrm{for all $\varphi \in C_0^\infty(\Omega)$,}
    \end{equation}
    then $c_1=...=c_m = 0$. 
\end{lemma}
\begin{proof}Without loss of generality we may assume that $a_1<a_2<...<a_m$. 
    Suppose that $a_1,...,a_m \in \Omega$ and $c_1,...,c_m\in \mathbb{R}$ are chosen such that \eqref{eq:Ableitungalsmass} holds for a Radon measure $\mu$. Let $\delta > 0$ be such that  $a_1,...,a_m \in (-R+ \delta, R-\delta)$. Define $\Omega' := (-R + \delta ,R- \delta)$ and observe that $\mu$ is a finite signed measure on $\Omega'$. Then one can use Fubini's theorem as follows: For each $\varphi \in C_0^\infty(\Omega')$ one has 
    \begin{align*}
        \int_{\Omega'} \varphi \; \mathrm{d}\mu & = \int_{(-R+\delta,R-\delta)} \left( \int_{(-R+\delta,x)} \varphi'(s) \; \mathrm{d}s \right) \; \mathrm{d}\mu(x)  = \int_{(-R+\delta,R-\delta)} \left( \int_{(s,R-\delta)}  \; \mathrm{d}\mu(x) \right) \varphi'(s) \; \mathrm{d}s
         \\ & = \int_{(-R+\delta,R-\delta)} \mu((s,R-\delta)) \varphi'(s) \; \mathrm{d}s.
    \end{align*}
    This and \eqref{eq:Ableitungalsmass} imply that for each $\varphi \in C_0^\infty(\Omega')$ one has 
    \begin{equation}\label{eq:42}
        \sum_{i = 1}^m c_i \varphi'(a_i) = \int_{(-R+\delta,R-\delta)} \mu((s,R-\delta)) \varphi'(s) \; \mathrm{d}s = \int_{\Omega'} f(s) \varphi'(s) \; \mathrm{d}s,
    \end{equation}
    where $f(s) := \mu((s,R-\delta))$ defines a function in $L^\infty(\Omega').$
    Now choose any function $\eta \in C_0^\infty(\Omega')$ such that $\int_{\Omega'} \eta = 1$. Notice that for each $\psi \in C_0^\infty(\Omega')$ the function $\varphi: \Omega' \rightarrow \mathbb{R},$  $\varphi(x) := \int_{-R+\delta}^x \psi(s) \; \mathrm{d}s - \int_{-R + \delta}^x  \eta(s) \; \mathrm{d}s \left( \int_{\Omega'} \psi(y) \; \mathrm{d}y \right) $ lies in $C_0^\infty(\Omega')$. Using this as a test  function in \eqref{eq:42} we find for any $\psi \in C_0^\infty(\Omega')$
    \begin{equation*}
        \sum_{i = 1}^m c_i \left( \psi(a_i) - \eta(a_i) \int_{\Omega'} \psi(y) \; \mathrm{d}y \right) = \int_{\Omega'} f(s) \left( \psi(s) - \eta(s) \int_{\Omega'} \psi(y) \; \mathrm{d}y \right) \; \mathrm{d}s. 
    \end{equation*}
    Rearranging terms we find
    \begin{equation*}
        \sum_{i = 1}^m c_i \psi(a_i) - \int_{\Omega'} \left( \sum_{i = 1}^m c_i \eta(a_i) \right) \psi(y)  \; \mathrm{d}y = \int_{\Omega'} f(s) \psi(s) \; \mathrm{d}s - \int_{\Omega'} \left(  \int_{\Omega'} f(s) \eta(s) \; \mathrm{d}s \right) \psi(y) \; \mathrm{d}y. 
    \end{equation*}
    Defining $g(s) := f(s)  +  \sum_{i = 1}^m c_i \eta(a_i) - \int_{\Omega'} f(y) \eta(y) \; \mathrm{d}y $ (switching the names of the variables $y$ and $s$ in the last summand) we obtain that $g \in L^\infty(\Omega')$ and
    \begin{equation}\label{eq:45}
        \sum_{i = 1}^m c_i \psi(a_i) = \int_{\Omega'} g(s) \psi(s) \; \mathrm{d}s \qquad \textrm{for all $\psi \in C_0^\infty(\Omega')$.}
    \end{equation}
    Considering test functions $\psi \in C_0^\infty((-R+\delta,a_1))$ we infer that $g \vert_{(-R+\delta,a_1)} = 0$ almost everywhere. Next, for $i=1,...,m-1$, using test functions $\psi \in C_0^\infty((a_i,a_{i+1}))$ yields $g \vert_{(a_i,a_{i+1})} = 0$ almost everywhere. Finally, considering  test functions $\psi \in C_0^\infty((a_m,R-\delta))$ we infer that $g \vert_{(a_m,R-\delta)} = 0$ almost everywhere. Since $\{a_1,...,a_m \}$ is finite we conclude that $g \equiv 0$ almost everywhere and in view of \eqref{eq:45}  we find 
    \begin{equation*}
        \sum_{i = 1}^m c_i \psi(a_i) = 0 \qquad \textrm{for all $\psi \in C_0^\infty(\Omega')$}. 
    \end{equation*}
    As a result one readily obtains $c_1=...=c_m = 0$ as stated. 
\end{proof}

\subsection{Regularity of measure-valued solutions}

\begin{lemma}
    Suppose that $\Omega \subset \mathbb{R}^n$ is a smooth domain and $v \in H^2(\Omega)$ satisfies \eqref{eq:measvaleq}. Then $v \in W^{3,q}_{loc}(\Omega)$ for any $ q \in (1,\frac{n}{n-1})$. 
\end{lemma}
\begin{proof}
    Let $\Omega' \subset\subset \Omega$ be a smooth subdomain and fix $q \in (1,\frac{n}{n-1})$. Then $\mu$ is a finite signed measure on $\Omega'$. By \cite[Proposition 5.1]{Ponce} there exists some $w \in W^{1,q}_0(\Omega')$ such that 
    \begin{equation*}
        \int_{\Omega'} \nabla w \cdot \nabla \varphi \; \mathrm{d}x = \int_{\Omega'} \varphi \; \mathrm{d}\mu \qquad \textrm{for all $\varphi \in C_0^\infty(\Omega')$.}
    \end{equation*}
    In particular, integration by parts implies 
    \begin{equation*}
       \int_{\Omega'} w  \Delta \varphi \; \mathrm{d}x = - \int_{\Omega'} \varphi \; \mathrm{d}\mu \qquad \textrm{for all $\varphi \in C_0^\infty(\Omega')$.}
    \end{equation*}
   Notice that by \eqref{eq:measvaleq} we have 
    \begin{equation*}
    \int_{\Omega'} \Delta v \Delta \varphi \; \mathrm{d}x = \int_{\Omega'} \varphi \; \mathrm{d}\mu \qquad \textrm{for all $\varphi \in C_0^\infty(\Omega')$},
\end{equation*}
Adding the previous two equations we find 
\begin{equation*}
    \int_{\Omega'} (\Delta v + w ) \Delta \varphi \; \mathrm{d}x = 0 \qquad \textrm{for all $\varphi \in C_0^\infty(\Omega')$},
\end{equation*}
that is, $\Delta v+ w$ is harmonic on $\Omega'$ and therefore lies in  $C^\infty(\Omega')$. Since $w \in W^{1,q}(\Omega')$ we infer that $\Delta v \in W^{1,q}_{loc}(\Omega')$. Elliptic regularity yields $v \in W^{3,q}_{loc}(\Omega')$. Since $\Omega$ can be exhausted by domains with smooth boundary we have $v \in W^{3,q}_{loc}(\Omega)$. 
\end{proof}
Now we show that our global minimizer found defined in \eqref{eq:2} does not lie in $W^{3,q}_{loc}(\Omega)$ for any $q \geq 1$.
\begin{lemma}
    Let $u$ be defined as in \eqref{eq:2} and $\Omega =B:=B_1(0)$. Then $u \in W^{2,\infty}(\Omega)$ but $u \not \in W^{3,q}_{loc}(\Omega)$ for any $q \geq 1$. 
\end{lemma}
\begin{proof}
    One readily checks that $u \in W^{2,\infty}(\Omega)$ and for a.e. $x \in B_1(0)$ we have
    \begin{equation*}
        D^2 u (x) = \mathrm{diag}(0,...,0,\chi_{B^+}(x)).
    \end{equation*}
    Assume now that $u \in W^{3,q}_{loc}(\Omega)$ for some $q \geq 1$. Then, in particular $\partial_{nn} u = \chi_{B^+} \in W^{1,q}_{loc}(\Omega)$ for some $q \geq 1$. This cannot be true (e.g. since Sobolev functions must be continuous along almost every ray $\{x_1 = s \} \cap \Omega$, cf. \cite[Theorem 4.21]{EvansGariepy}). 
\end{proof}

In \cite[Example 2 and 4]{GMMathAnn} it was shown for $n=2$ that for a suitably small constant boundary datum $u_0 > 0$ minimizers $u$ of $\mathcal{F}(\cdot, B_1(0))$ under Dirichlet (respectively Navier) boundary conditions are unique and satisfy for almost every $x \in B_1(0)$
\begin{equation}\label{eq:52}
    \Delta u(x) = \begin{cases} 0 & |x| < \rho  \\
        C\log|x| + D  & |x| > \rho 
    \end{cases},
\end{equation}
for some $C,D \in \mathbb{R}$ and $\rho \in (0,1)$. 
Moreover, \cite[Remark 9 and Remark 11]{GMMathAnn} show that $\Delta u \not \in C^0(B_1(0)),$ that is $C \log(\rho)+ D \neq 0$. The following lemma shows that in this case $u \not \in W^{3,1}_{loc}(B_1(0))$.

\begin{lemma}
    Suppose that $u \in H^2(B_1(0))$ satisfies \eqref{eq:52} for some $C,D \in \mathbb{R}$ and $\rho \in (0,1)$. Then $u \in W^{3,1}_{loc}(B_1(0))$ if and only if $C \log(\rho) + D= 0$.
\end{lemma}
\begin{proof}
   It is an immediate consequence of \cite[Theorem 4.21]{EvansGariepy} that Sobolev functions must also be continuous along almost every ray $\{ r \sigma : r \in (0,1) \}$ $(\sigma \in \partial B_1(0))$. If we pick one such ray and use \eqref{eq:52} we infer that $C \log \rho + D = 0$. The opposite direction is an easy computation and left to the reader.  
\end{proof}



\end{document}